\documentclass{proc-l}
\usepackage{graphicx}
\copyrightinfo{2016}{American Mathematical Society}

\newtheorem{theorem}{Theorem}[section]
\newtheorem{lemma}[theorem]{Lemma}

\theoremstyle{definition}
\newtheorem{definition}[theorem]{Definition}

\theoremstyle{remark}

\numberwithin{equation}{section}
\def\rrho{\mbox{\boldmath$\rho$}}
\def\rrhoh{\mbox{\boldmath$\hat\rho$}}
\def\x{\textit{\textbf{x\/}}}
\def\k{\textit{\textbf{k\/}}}

\def\C{\mathbb{C}}
\def\Oe{O_{\varepsilon}}
\def\cj{c_j^{}}
\def\Erho{E_\rho^{}}
\def\dE{d_E^{}}
\def\dT{d_T^{}}
\def\dmax{d_{\max{}}^{}}

\begin{document}

\title[Approximation in the hypercube]
{Multivariate polynomial approximation in the hypercube}

\author[Trefethen]{Lloyd N. Trefethen}
\address{Mathematical Institute, University of Oxford,
Oxford, OX2 6GG, UK}
\curraddr{}
\email{trefethen@maths.ox.ac.uk}
\thanks{Supported by the European Research Council under the European
Union's Seventh Framework Programme (FP7/2007--2013)/ERC grant
agreement no.\ 291068.\ \ The views expressed in this article
are not those of the ERC or the European Commission, and the
European Union is not liable for any use that may be made of the
information contained here.}


\commby{}

\begin{abstract}
A theorem is proved concerning approximation of analytic
functions by multivariate polynomials in the $s$-dimensional
hypercube.  The geometric convergence rate is determined not
by the usual notion of degree of a multivariate polynomial,
but by the {\em Euclidean degree,} defined in terms of the
2-norm rather than the 1-norm of the exponent vector $\k$ of
a monomial $x_1^{k_1}\cdots \kern .8pt x_s^{k_s}$.
\end{abstract}

\maketitle

\section{Introduction}

The aim of this paper is to prove a theorem concerning an
effect identified in Section 6 of~\cite{iso}.  If an analytic
function $f(\x) = f(x_1^{},\dots,x_s^{})$ is approximated
by multivariate polynomials in the $s$-dimensional hypercube
$[-1,1]^s$, the usual notion of polynomial degree, namely the
total degree, is not the right predictor of approximability.
In the hypercube, the set of polynomials of a given total
degree has $\sqrt s$ times finer resolution along a direction
aligned with an axis than along a diagonal.  Conversely, the
set of polynomials of a given maximal degree has $\sqrt s$
times finer resolution along a diagonal than along an axis.
To achieve balanced resolution in all directions one should
work with polynomials of a given {\em Euclidean degree.}

Our definitions are as follows.  With $s\ge 1$,
we consider functions
$f(\x)= f(x_1^{},\dots,x_s^{})$ in $[-1,1]^s$, with 
$\|\cdot \|_{[-1,1]^s}^{}$ representing the maximum norm over
this set.
For a monomial $x_1^{k_1}\cdots \kern 1pt x_s^{k_s}$ we define
\begin{eqnarray}
\mbox{\em Total degree}: &~~ & \dT = \|\k\|_1^{}, \\[2pt]
\mbox{\em Euclidean degree}: &~~ & \dE =
   \|\k\|_2^{}, \\ [2pt]
\mbox{\em Max degree}: &~~ & \dmax = \|\k\|_\infty^{},
\end{eqnarray}
where $\|\cdot\|_1^{}$, $\|\cdot\|_2^{}$, and
$\|\cdot\|_\infty^{}$ are
the 1-, 2-, and $\infty$- norms
of the $s$-vector $\k = (k_1^{}\dots, k_s^{})$,
and the degree of a multivariate polynomial is the maximum of
the degrees of its nonzero monomial constituents.  The total
and maximal degree definitions are standard and appear
in publications like~\cite{mason} and~\cite{timan} where multivariate
polynomial approximation in the hypercube is discussed,
but the Euclidean degree seems to be new in~\cite{iso}.
Note that $\dE$ is not in general an integer.

Our interest is in leading order exponential effects, not
algebraic fine points, and accordingly, we will make use of the notation
$\Oe$ defined as follows: $g(n) = \Oe(a^n)$ if for all
$\varepsilon>0$, $g(n) = O((a+\varepsilon)^n)$ as $n\to\infty$.
By $\Oe(a^{-n})$ we mean $\Oe((1/a)^n)$, or equivalently,
for all
$\varepsilon>0$, $O((a-\varepsilon)^{-n})$.

\section{Numerical illustration}
The case $s=2$ suffices for a numerical illustration.
Let $f$ be the 2D Runge function
\begin{equation}
f(x,y) = \frac{1}{1+ 10(x^2+y^2)},
\label{runge}
\end{equation}
which is analytic for all real values of $x$ and $y$ and isotropic
in the sense that it is invariant with respect to rotation in
the $x$-$y$ plane.  Figure~\ref{fig} gives an indication of the
minimal error in approximation of $f$ on $[-1,1]^2$ by bivariate
polynomials of various total, Euclidean, and maximal degrees.
(Bivariate Chebyshev coefficients of~$f$ are plotted in
Figure~6.4 of~\cite{iso}.)
The figure is actually based on $L^2$ rather than $L^\infty$ 
approximations, since these are much easier to compute,
but this is enough to give an indication of the
separation between the convergence rates when the degree
is defined by $\dT$ and when it is defined by $\dE$ or $\dmax$.
\begin{figure}
\begin{center}
~\vskip .1in
\includegraphics[scale=.7]{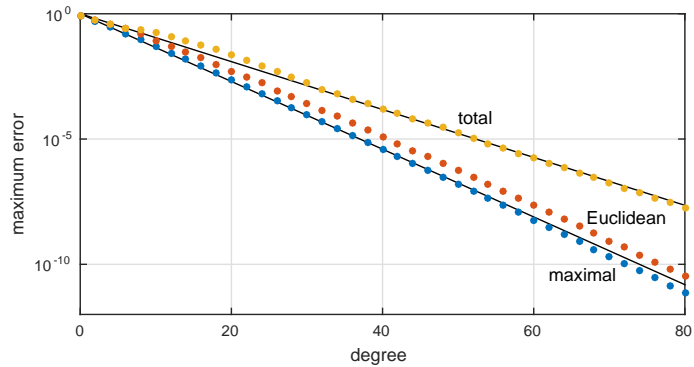}
\end{center}
\caption{Maximum-norm errors in 
approximation of the Runge function\/
$(\ref{runge})$ as a function of degree $n$ in the unit square, for
three different definitions of degree.
The approximations come from least-squares minimization over a
fine grid in $[-1,1]^2$.  Straight lines mark the convergence rates
of Theorem\/~$\ref{thm1}$.}
\label{fig}
\end{figure}

The function (\ref{runge}) satisfies Assumption A of our theorem,
Theorem~\ref{thm1}, with $h^2=0.1$, and the
data in the figure show convincing agreement with
the predictions of the theorem.  This function is analytic when
$x$ and $y$ are real but not when they are complex.  On the other
hand the similar function
\begin{equation}
g(x,y) = \frac{1}{21-10(x^2+y^2)}
\label{runge2}
\end{equation}
has real singularities just outside the unit square.
Theorem~\ref{thm1} applies with
$h^2=0.1$ for this function too, and a plot of convergence
rates (not shown) looks almost exactly like Figure~\ref{fig}.

\section{Chebyshev series in 1D}

In the standard theory for a single variable $x$, for any 
$\rho>1$, let $\Erho$ denote the open set bounded by the
Bernstein $\rho\kern .7pt$-ellipse in the complex $x$-plane, i.e., the image
of the circle $|w|=\rho$ under the map $x = (w+w^{-1})/2$.
A Lipschitz continuous function $f$ defined on
$[-1,1]$ has an absolutely and uniformly
convergent Chebyshev series
\begin{equation}
f(x) = \sum_{k=0}^\infty a_k^{} T_k^{}(x),
\end{equation}
where $T_k^{}$ is the Chebyshev polynomial of degree $k$.
Truncating the series at degree~$n$ gives the polynomial approximation
\begin{equation}
p_n^{}(x) = \sum_{k=0}^n a_k^{} T_k^{}(x).
\end{equation}
The following result goes back to Bernstein's prize-winning
memoir of 1912~\cite{bernstein}.
Here and elsewhere, when we say that $f$ is analytic in a region,
we mean that if necessary $f$ can be analytically continued to that region.

\begin{lemma}
\label{lemma1}
If\/ $f$ is analytic in $\Erho$, its Chebyshev coefficients
and truncated Chebyshev expansions satisfy 
\begin{equation}
a_k^{} = \Oe(\kern .7pt\rho^{-k}), \quad
\|f-p_n^{}\|_{[-1,1]}^{} = \Oe(\kern .7pt\rho^{-n}).
\label{rhoests}
\end{equation}
\end{lemma}

\begin{proof}
The second estimate follows from the first, 
whose proof can be based on contour integrals over
$\tilde\rho\kern .7pt$-ellipses with $\tilde \rho= \rho-\varepsilon$
for arbitrarily small $\varepsilon >0$,
or equivalently on contour integrals over
circles in the $z$-plane after a change of variables from $x\in
[-1,1]$ to $z$ on the unit circle.  See Theorems 8.1 and 8.2
of~\cite{atap}.
\end{proof}

For our purposes it will be important to consider $x^2$ as well as
$x$.  When $x$ ranges over $\Erho$, with foci $-1$ and $1$ and
topmost point $ih$, $x^2$ ranges
over another ellipse, with foci $0$ and $1$ and leftmost point
$-h^2$, where $h$ and $\rho$ are related by
\begin{equation}
h = (\kern .7pt\rho-\rho^{-1})/2 , \quad
\rho = h + \sqrt{1+h^2}.
\label{correspondence}
\end{equation}
Arnol'd calls $\Erho$ a {\em Hooke ellipse} and $E_\rho^2$ a {\em
Newton ellipse}~\cite{arnold}.
We wish to parametrize the latter by $h^2$ rather than $\rho$,
so we make the following definition.

\begin{definition}
For any $s,a>0$,
$N_{s,a}^{}$ is the open region in the complex plane bounded by
the ellipse with foci $0$ and $s$ and leftmost point $-a$.
Equivalently, it is the region consisting of points $x$ satisfying
$|x| + |x-s| < s + 2a$.
\end{definition}

\noindent
Thus $E_\rho^2 = N_{1,h^2}^{}$, and
Lemma~\ref{lemma1} can be equivalently restated as follows.

\begin{lemma}
\label{cor1}
Suppose that for some $h>0$, $f(x)$ is analytic for all\/ $x\in\C$
such that $x^2\in N_{1,h^2}$.  Then\/ $(\ref{rhoests})$ holds
with\/ $\rho = h + \sqrt{1+h^2}$.
\end{lemma}

\section{Main theorem}
Now let $f$ be a function of $\x\in [-1,1]^s$ for some
$s\ge 1$.  If $f$ is smooth,
it has a uniformly and absolutely
convergent multivariate Chebyshev series
\begin{equation}
p(\x) = \sum_{k_1=0}^\infty \cdots \sum_{k_s=0}^\infty
a_{k_1,\dots,k_s} T_{k_1}^{}\kern -1pt (x_1) \cdots
T_{k_s}^{}\kern -1pt  (x_s)
\label{chebs}
\end{equation}
(see e.g.\ Theorem 4.1 of~\cite{mason}).
Here is our analyticity assumption, generalizing that 
of Lemma~\ref{cor1}.

\medskip

\noindent{\em {\bf Assumption A.}
For some $h>0$, $f(\x)$ is analytic for all\/ $\x\in \C^s$ in the
$s$-dimensional region defined by the condition
$x_1^2+\cdots + x_s^2 \in N_{s,h^2}^{}$.
}

\medskip

\noindent Note that a sufficient condition for
Assumption A to hold is that $f(\x)$ is analytic for all $\x$
with $\Re (x_1^2 + \cdots + x_s^2) > -h^2$.

The following lemma will be proved in the next section.

\begin{lemma}
\label{lemma2}
If\/ $f$ satisfies Assumption A, its multivariate Chebyshev coefficients
satisfy
\begin{equation}
a_{\k}^{} = \Oe(\kern .7pt\rho^{-\|\k\|_2}),
\label{multicoeffs}
\end{equation}
where\/ $\rho = h + \sqrt{1+h^2}$.
\end{lemma}

Based on this result,
our theorem bounds the convergence rates
of polynomial approximations defined by total, Euclidean,
and maximal degree.

\begin{theorem}
\label{thm1}
If\/ $f$ satisfies Assumption A, then
$$
\inf_{d(\kern .4pt p)\le n}
\| f-p^{}\|_{[-1,1]^s}^{} 
= \begin{cases}
\Oe(\kern .7pt\rho^{-n/\sqrt s}) & \hbox{if\/ } d=\dT, \\
\Oe(\kern .7pt\rho^{-n}) & \hbox{if\/ } d=\dE, \\
\Oe(\kern .7pt\rho^{-n}) & \hbox{if\/ } d=\dmax, 
\end{cases}
$$
where\/ $\rho = h + \sqrt{1+h^2}$.
\end{theorem}

{\em Proof of Theorem\/~$\ref{thm1}$, assuming Lemma\/~$\ref{lemma2}$.}
The second (middle) assertion of the theorem follows from Lemma~\ref{lemma2}
by truncating the multivariate Chebyshev series (\ref{chebs}), since
$|T_{k_1}(x_1^{})\cdots T_{k_s}(x_s^{})| \le 1$ for all $\k$ for
all $\x\in [-1,1]^s$.  The third assertion is a consequence
of the second, since $\dmax(\kern .7pt p)\le \dE(\kern .7pt p)$ for
any multivariate polynomial~$p$.  The first assertion is also
a consequence of the second since $\dT(\kern .7pt p)/\sqrt s
\le \dE(\kern .7pt p)$.

\section{Proof of Lemma \ref{lemma2}}
To complete the proof of Theorem~\ref{thm1}
we must prove Lemma~\ref{lemma2}.\ \ For this we 
will make use of a result in the book by Bochner and Martin~\cite{bochmart}.
Let $\rrho = (\kern .7pt \rho_1^{}, \dots, \rho_s^{})$ be an
$s$-vector with $\rho_j^{} > 1$ for each $j$, and let
$E(\kern .7pt\rrho)\subset \C^s$ be the {\em elliptic polycylinder}
defined as the set of all
points $\x\in \C^s$ such that $x_j^{}\in E_{\rho_j}$ for each~$j$.
The result in question is an $s$-dimensional generalization of Lemma~\ref{lemma1}.

\begin{lemma}
\label{lemma3}
Let $f$ be analytic in $E(\kern .7pt\rrho)$.
Then its multivariate Chebyshev coefficients satisfy
\begin{equation}
a_{\k}^{} = \Oe(\kern .7pt \rho_1^{-k_1}\cdots \rho_s^{-k_s})
\label{eq12}
\end{equation}
as $k_1^{}+\dots + k_s^{} \to\infty$. 
\end{lemma}

\begin{proof}
Equation (\ref{eq12}) means that for any
$\varepsilon>0$,
$a_{\k}^{} = O((\kern .5pt \rho_1^{}-\varepsilon)^{-k_1^{}}
\cdots (\kern .5pt \rho_s^{}-\varepsilon)^{-k_s^{}})$.  This is
essentially Theorem~11 on p.~95 of~\cite{bochmart},
which is derived by contour integrals.  For further
discussion see~\cite{boyd}.
\end{proof}

\begin{proof}[Proof of Lemma\/~$\ref{lemma2}$.]
For any $s$-vector $\k$ of nonnegative indices, define
$$
\cj = \frac{k_j}{\|\k\|_2^{}} \le 1
$$
and
$$
h_j^{} = \cj h .
$$
Then $h_1^2 + \cdots + h_s^2 = h^2$, so by 
Assumption A, $f(\x)$ is analytic in the subset of $\C^s$
defined by the condition $x_1^2+\cdots + x_s^2 \in N_{s,h^2}^{}$.
From Lemma~\ref{minklem} below, we have
\begin{equation}
N_{1,h_1^2}^{} \oplus
\cdots \oplus N_{1,h_s^2}^{} \subseteq N_{s,h_1^2+\cdots h_s^2}^{},
\end{equation}
where $\oplus$ denotes the standard Minkowski sum of sets.  It follows
that $f(\x)$ is analytic whenever $x_j^{}\in N_{1,h_j^2}$ for
each $j$.  In other words,
$f(\x)$ is analytic in the elliptic polycylinder
$E(\kern .7pt\rrhoh)$ with $\hat\rho_j^{}$ defined by 
$$
\hat\rho_j^{} = h_j^{} + \sqrt{1 + h_j^2}
= \cj h + \sqrt{1 + (\cj h)^2}.
$$
It can be shown (Lemma~\ref{lemma4}, below)
this this final quantity is greater than or equal to the number 
$\rho_j^{}$ which we define by
$$
\rho_j^{} = \big(h + \sqrt{1 + h^2}\,\big)^{\cj} =
\rho^{k_j^{}/\|\k\|_2^{}}.
$$
Therefore if $\rrho$ is the $s$-vector
with components given by this formula, then
the associated polycylinder satisfies $E(\kern .7pt\rrho)\subseteq E(\kern .7pt\rrhoh)$, and 
$f(\x)$ is analytic in $E(\kern .7pt\rrho)$.
We now calculate
\begin{equation}
\rho_1^{-k_1}\cdots \rho_s^{-k_s}
= \rho^{-(k_1^2 + \cdots +k_s^2)/\|\k\|_2^{}} = \rho^{-\|\k\|_2^{}},
\label{eq15}
\end{equation}
and inserting this identity in 
(\ref{eq12}) gives (\ref{multicoeffs}), as required.
\end{proof}

Here are the two lemmas just used.

\begin{lemma}
\label{minklem}
For any $s,t>0$ and $a,b>0$,
\begin{equation}
N_{s,a}^{} \oplus N_{t,b}^{} \subseteq N_{s+t,a+b}^{} .
\end{equation}
\end{lemma}

\begin{proof}
If $x \in N_{s,a}^{}$ and $y\in N_{t,b}^{}$, then we have
\begin{equation*}
|x| + |x-s| < s+2a, \quad |y| + |y-t| < t+2b.
\end{equation*}
Therefore by the triangle inequality,
\begin{equation*}
|x+y| + |(x-s)+(y-t)| < (s+2a)+(t+2b),
\end{equation*}
that is,
\begin{equation*}
|x+y| + |(x+y)-(s+t)| < (s+t) + 2(a+b),
\end{equation*}
which implies $x+y \in N_{s+t,a+b}^{}$.
\end{proof}

\begin{lemma}
\label{lemma4}
For any $h\ge 0$ and $c\in [\kern .5pt 0,1]$,
$c\kern .4pt h + \big(1 + c^2h^2\big)^{1/2} \ge \big(h + (1 + h^2)^{1/2}\,\big)^c$.
\end{lemma}

\begin{proof}
Given $h$, define $\varphi(c) = c\kern .4pt h + (1 + c^2h^2)^{1/2}$.  We must show
$\varphi(c) \ge \varphi(1)^c$, or equivalently
\begin{equation}
\psi(c) \ge c \kern 1pt \psi(1), \quad 0\le c \le 1,
\label{convex}
\end{equation}
where
\begin{equation*}
\psi(c) = \log(\varphi(c)) = 
\log\big( c\kern .4pt h + (1 + c^2h^2\kern 1pt )^{1/2} \big) .
\end{equation*}
Since $\psi(0) = 0,$ a sufficient condition for (\ref{convex}) to hold
is that $\psi$ is convex in the sense that 
$\psi''(c)\le 0$ for $c\in [\kern .5pt 0,1]$.
This follows from the identity
$\psi''(c) = - c\kern .4pt h^3\big( 1+ c^2h^2)^{-3/2}$.
\end{proof}

\section{Discussion}
This work is motivated by computational applications, since
for computation in higher dimensions, a hypercube is usually
the domain of choice.  Theorem~\ref{thm1} suggests that any
method for computation in a hypercube that is based on one of the
familiar definitions of the degree of a multivariate polynomial, namely
total degree or maximal degree,
is likely to be suboptimal.  For example, a standard idea of
multidimensional quadrature (cubature) is the exact
integration of multivariate polynomial approximations of
a given total degree, an idea going back to Maxwell~\cite{ku,maxwell}.
For functions $f$ with approximately angle-independent complexity
in the hypercube,
the suboptimality factors may be exponentially large
as $s\to\infty$, since
the number of degrees of freedom increases by
$O((\sqrt s\kern 1.3pt)^s)$
when the degree increases by $\sqrt s$.
For example, in $s=10$ dimensions, our results suggest that
isotropic analytic functions can be resolved to a given accuracy by
polynomials of fixed Euclidean degree with
$11.1$ times fewer degrees of freedom than
polynomials of fixed total degree, and
$401.5$ times fewer degrees of freedom than
polynomials of fixed maximal degree.
We shall not give details
here, since these are discussed in Section~6 of~\cite{iso} with
formulas and a table to quantify the exponential effects.

Theorem~\ref{thm1} is only an upper bound, so in principle, the
difference it suggests between $\dT$ and the other degrees
$\dE$ and $\dmax$ might be illusory.  However, numerical
experiments such as that of Figure~\ref{fig} and those reported
in~\cite{iso} make
it clear that the difference is genuine.  This could
be made rigorous by the development of a converse theorem,
as has been long established in the 1D case, again thanks to Bernstein
(see Theorem~8.3 of~\cite{atap}).  

In closing I would like to highlight a conceptual link
between this note and my earlier paper~\cite{ht} with Nick Hale.
The central observation of~\cite{ht} is that the resolving power
of (univariate) polynomials on an interval $[-1,1]$ is nonuniform, 
making polynomials fall short of optimality by a factor of $\pi/2$
in representing functions whose complexity on $[-1,1]$
is uniform.  In the present work, the issue is again
nonuniformity of polynomials, but now they are
multivariate and the uniformity issue pertains
to rotation rather than translation.  
As pointed out in Section~7 of~\cite{iso},
the translational issue is present in multiple dimensions too.

\section*{Acknowledgments}
I have benefitted from extensive discussions
of approximation and cubature in the hypercube
with Hadrien Montanelli and Klaus Wang.  I thank also Jared Aurentz,
Stefan G\"uttel, Nick Hale,
Allan Pinkus, and Alex Townsend for helpful suggestions.  It was
Aurentz who proposed the term ``Euclidean degree.''

\bibliographystyle{amsplain}

\end{document}